\documentclass[11pt,twoside]{amsart}
\usepackage{latexsym,amssymb,amsmath}
\usepackage[all]{xy}
\usepackage{tikz,pgfplots}
\usepackage{extpfeil}
\usepackage{afterpage}
\usepackage{float}
\usepackage{hyperref}
\usepackage{relsize}

\makeatletter  % for 2020 subject classification
\@namedef{subjclassname@2020}{%
\textup{2020} Mathematics Subject Classification}
\makeatother

\numberwithin{equation}{section}
\newtheorem{theorem}{Theorem}[section]
\newtheorem{lemma}[theorem]{Lemma}
\newtheorem{proposition}[theorem]{Proposition}
\newtheorem{corollary}[theorem]{Corollary}

\theoremstyle{definition}
\newtheorem{definition}[theorem]{Definition} 
\newtheorem{procedure}[theorem]{Procedure} 

\newtheorem{remark}[theorem]{Remark}
\newtheorem{example}[theorem]{Example}

\begin{document}

\title[Induced matchings and the \text{v}-number of graded ideals]
{Induced matchings and the \text{v}-number of graded ideals}
\thanks{The first author was supported by a scholarship from CONACYT,
Mexico. The second and third authors were supported by SNI, Mexico.}

\author[G. Grisalde]{Gonzalo Grisalde}
\address{
Departamento de
Matem\'aticas\\
Centro de Investigaci\'on y de Estudios
Avanzados del
IPN\\
Apartado Postal
14--740 \\
07000 Mexico City, Mexico
}
\email{gjgrisalde@math.cinvestav.mx}

\author[E. Reyes]{Enrique Reyes}
\address{
Departamento de
Matem\'aticas\\
Centro de Investigaci\'on y de Estudios
Avanzados del
IPN\\
Apartado Postal
14--740 \\
07000 Mexico City, Mexico
}
\email{ereyes@math.cinvestav.mx}

\author[R. H. Villarreal]{Rafael H. Villarreal}
\address{
Departamento de
Matem\'aticas\\
Centro de Investigaci\'on y de Estudios
Avanzados del
IPN\\
Apartado Postal
14--740 \\
07000 Mexico City, Mexico
}
\email{vila@math.cinvestav.mx}
%\thanks{*Corresponding author}


\dedicatory{Dedicated to the memory of Wolmer Vasconcelos}  

\begin{abstract} 
We give a formula for the v-number of a
graded ideal that can be used to compute this number. Then we show that
for the edge ideal $I(G)$ of a graph $G$ the induced matching number of
$G$ is an upper
bound for the v-number of $I(G)$ when $G$ is very well-covered, or $G$ 
has a simplicial partition, or $G$ is well-covered connected and
contain neither $4$- nor $5$-cycles. In all these cases the v-number
of $I(G)$ is a lower bound for the regularity of the 
edge ring of $G$. 
We classify when the upper bound holds when $G$ is a
cycle, and classify when all vertices of a graph are shedding
vertices to gain insight on $W_2$-graphs.
\end{abstract}

\maketitle 

\section{Introduction}\label{intro-section}
Let $S=K[t_1,\ldots,t_s]=\bigoplus_{d=0}^{\infty} S_d$ be a polynomial ring over
a field $K$ with the standard grading and let $I$ be a graded ideal of
$S$. A prime ideal $\mathfrak{p}$ of $S$ is an \textit{associated prime} of $S/I$ if
$(I\colon f)=\mathfrak{p}$ for some $f\in S_d$, where $(I\colon f)$ is
the set of all $g\in S$ such that $gf\in I$. The set of associated primes of $S/I$ 
is denoted by ${\rm Ass}(I)$ and set of maximal elements of ${\rm Ass}(I)$ with respect to inclusion is
denoted by ${\rm Max}(I)$. The v-{\em number} of $I$, denoted ${\rm
v}(I)$, is the following invariant of
$I$ that was introduced in \cite{min-dis-generalized} to study the asymptotic behavior 
of the minimum distance of projective Reed--Muller-type codes
\cite[Corollary~4.7]{min-dis-generalized}:
$$
{\rm v}(I):=\min\{d\geq 0 \mid\exists\, f 
\in S_d \mbox{ and }\mathfrak{p} \in {\rm Ass}(I) \mbox{ with } (I\colon f)
=\mathfrak{p}\}.
$$
\quad One can define the v-number of $I$ locally at each associated 
prime $\mathfrak{p}$ of $I$\/:
$$
{\rm v}_{\mathfrak{p}}(I):=\mbox{min}\{d\geq 0\mid \exists\, f\in S_d \mbox{ with }(I\colon f)=\mathfrak{p}\}.
$$
\quad For a graded module $M\neq 0$, we define 
$\alpha(M):=\min\{\deg(f) \mid f\in M\setminus\{0\}\}$. By convention, 
we set $\alpha(0):=0$. Part (d) of the next result was shown in
\cite[Proposition~4.2]{min-dis-generalized} for unmixed graded
ideals. The next result gives a formula for the v-number of any graded
ideal.

\noindent \textbf{Theorem~\ref{vnumber-general}.}\textit{ 
Let $I\subset S$ be a graded ideal and let $\mathfrak{p}\in{\rm
Ass}(I)$. The following hold.
\begin{enumerate}
\item[(a)]
If $\mathcal{G}=\{\overline{g}_1,\ldots,\overline{g}_r\}$ is
a homogeneous minimal generating set of $(I\colon\mathfrak{p})/I$,
then  
$$
{\rm v}_{\mathfrak{p}}(I)=
\min\{\deg(g_i)\mid 1\leq i\leq r\mbox{ and }(I\colon
g_i)=\mathfrak{p}\}. 
$$
\item[(b)] ${\rm v}(I)=\min\{{\rm v}_{\mathfrak{q}}(I)\mid
\mathfrak{q}\in{\rm Ass}(I)\}$. 
\item[(c)] ${\rm
v}_{\mathfrak{p}}(I)\geq\alpha((I\colon\mathfrak{p})/I)$ with equality
if $\mathfrak{p}\in{\rm Max}(I)$.
\item[(d)] If $I$ has no embedded primes, then $
{\rm v}(I)=\min\{\alpha\left((I\colon\mathfrak{q})/{I}\right)\vert\,
\mathfrak{q}\in{\rm Ass}(I)\}.
$
\end{enumerate}}

The formulas of parts (a) and (b) give an algorithm to compute the v-number
number using \textit{Macaulay}$2$ \cite{mac2} (Example~\ref{example1},
Procedure~\ref{procedure1}). 

The v-number of non-graded ideals was used in \cite{dual} to compute the
regularity index of the minimum distance function of affine
Reed--Muller-type codes \cite[Proposition~6.2]{dual}. In this case, 
one considers the vanishing ideal of a 
set of affine points over a finite field.

For certain classes of graded ideals ${\rm v}(I)$ is a lower bound for
${\rm reg}(S/I)$, the regularity of the quotient ring $S/I$
(Definition~\ref{regularity-def}), see
\cite{min-dis-generalized,v-number,footprint-ci}. There are examples
of ideals where ${\rm v}(I)>{\rm reg}(S/I)$ \cite{v-number}. It is an open problem whether 
${\rm v}(I)\leq {\rm reg}(S/I)+1$ holds for any squarefree
monomial ideal. Upper and lower bounds for the regularity of edge
ideals and their powers are given in 
\cite{banerjee-etal,Beyarslan-etal,Dao-Huneke-schweig,Herzog-Hibi-ub,JS,edge-ideals,woodroofe-matchings},
see Section~\ref{prelim-section}.
Using the polarization technique of Fr\"oberg \cite{Fro1}, we give an
upper bound for the regularity of a monomial ideal $I$ in terms of the
dimension of $S/I$ and the exponents of the monomials that 
generate $I$ (Proposition~\ref{reg-dim-pol}). 

Let $G$ be a graph with vertex set $V(G)$ and edge
set $E(G)$. If $V(G)=\{t_1,\ldots,t_s\}$, we can regard each vertex $t_i$ as a
 variable of the polynomial ring $S=K[t_1,\ldots,t_s]$ and think of each edge $\{t_i,t_j\}$ of
 $G$ as the quadratic monomial
 $t_it_j$ of $S$. The
 \textit{edge ideal} of $G$ is the squarefree monomial ideal of $S$ defined as 
$$I(G):=(t_{i}t_{j}\mid \{t_{i},t_{j}\}\in E(G)).$$
\quad This ideal, introduced in \cite{cm-graphs}, 
has been studied in the literature from 
different perspectives, see
\cite{graphs-rings,Herzog-Hibi-book,unmixed-c-m,ITG,monalg-rev} and the
references therein. We use induced matchings of $G$ to compare the v-number of
$I(G)$ with the regularity of $S/I(G)$ for certain families of graphs.

A subset $C$ of $V(G)$ is a {\it vertex cover\/} of $G$ if every edge
of $G$ is incident with at least one vertex in $C$. A vertex cover $C$ of $G$
is {\it minimal\/} if each proper subset of $C$ is not a vertex cover
of $G$. A subset $A$ of $V(G)$ is called {\it stable\/}
if no two points in $A$ are 
joined by an edge. Note that a set of vertices $A$ is a (maximal) 
stable set of $G$ if and only if
$V(G)\setminus 
A$ is a (minimal) vertex cover of $G$. The \textit{stability number}
of $G$, denoted by $\beta_0(G)$,  
is the cardinality of a maximum stable set of $G$ and the 
{\it covering number\/} of $G$, denoted $\alpha_0(G)$, is the
cardinality of a minimum vertex cover of $G$. For use below we introduce the following  two families of stable
sets:
\begin{align*}
\mathcal{F}_G&:=\{A\mid A\mbox{ is a maximal stable set 
of }G\},\mbox{ and}\\
\mathcal{A}_G&:=\{A\mid A\mbox{ is a stable set 
of }\, G\mbox{ and }N_G(A)\mbox{ is a minimal vertex
cover of }G\}.
\end{align*}

According to \cite[Theorem~3.5]{v-number}, 
$\mathcal{F}_G\subset\mathcal{A}_G$ and the $\mathrm{v}$-number of
$I(G)$ is given by
$$
\mathrm{v}(I(G)) = \min\{\vert A \vert : A\in\mathcal{A}_G\}. 
$$
\quad The v-number of $I(G)$ is a combinatorial invariant of $G$
that has been used to characterize the family of $W_2$-graphs (see
the discussion below before Corollary~\ref{w2-graph}). We can 
define the v-number of a graph $G$ as  ${\rm v}(G):={\rm v}(I(G))$ 
and study ${\rm v}(G)$ from the viewpoint of graph theory.

A set $P$ of pairwise disjoint edges of $G$ is called a
{\it matching\/}. A matching $P=\{e_1,\ldots,e_r\}$ is \textit{perfect} if
$V(G)=\bigcup_{i=1}^re_i$. An {\it induced matching\/} of a graph $G$ is a
matching $P=\{e_1,\ldots,e_r\}$ of $G$ such that the only edges of 
$G$ contained in $\bigcup_{i=1}^re_i$ are $e_1,\ldots,e_r$. 
The \textit{matching number} of $G$, denoted $\beta_1(G)$, 
is the maximum cardinality of a matching of $G$ and the
{\it induced matching number\/} of $G$,
denoted ${\rm im}(G)$, is the number of edges in the largest
induced matching. 

The graph $G$ is
{\it well-covered\/} 
if every maximal stable set of $G$ is of the same size 
and $G$ is  \textit{very
well-covered} if $G$ is well-covered, has no isolated vertices, 
and $|V(G)|=2\alpha_0(G)$. The class of very well-covered graphs
includes the bipartite well-covered graphs without isolated vertices \cite{ravindra,unmixed} and
the whisker graphs \cite[p.~392]{ITG} (Lemma~\ref{bipartite-whiskers}). A
graph without isolated vertices is very well-covered if and only
if $G$ is well-covered and $\beta_1(G)=\alpha_0(G)$
(Proposition~\ref{konig-vwc}). One of
the properties of very well-covered graphs that will be used
to show the following theorem is that they can be classified using combinatorial
properties of a perfect matching as was shown by Favaron \cite[Theorem~1.2]{favaron}
(Theorem~\ref{konig}, cf.~Theorem~\ref{lemma-Ivan}). 

We come to one of our main results.

\noindent \textbf{Theorem~\ref{Domi-InduceMatch}.}\textit{
Let $G$ be a very well-covered graph and let $P=\{e_1, \ldots, e_r\}$
be a perfect matching of $G$. 
Then, there is an induced submatching $P'$ of $P$ and $D \in {\mathcal A}_G$
such that $D \subset V(P')$ and $\vert e \bigcap D \vert = 1$ for
each $e\in P'$. 
Furthermore ${\rm v}(I(G))\leq|P'|=|D|\leq{\rm im}(G)\leq{\rm
reg}(S/I(G))$.
}

Let $G$ be a graph and let $W_G$ be its whisker graph 
(Section~\ref{prelim-section}).  
As a consequence we recover a result of \cite{v-number} showing that
the v-number of 
$I(W_G)$ is bounded
from above by the regularity of the quotient ring $K[V(W_G)]/I(W_G)$
(Corollary~\ref{sep29-21}).  The \textit{independent
domination number} of $G$, denoted by $i(G)$, is the minimum size of 
a maximal stable set \cite[Proposition~2]{Allan-Laskar}: 
$$i(G):=\min\{|A|\colon
A\in\mathcal{F}_G\},$$ 
and $i(G)$ is equal to the v-number of the whisker graph $W_G$ of $G$
\cite[Theorem 3.19(a)]{v-number}.

A cycle of length $s$ is denoted by $C_s$. 
The inequality ${\rm v}(I(G))\leq{\rm reg}(S/I(G))$ of
Theorem~\ref{Domi-InduceMatch}  
is false if we only assume that $G$ is a well-covered graph,
since the cycle $C_5$ is a well-covered graph, but one has ${\rm im}(C_5)=1 <
2={\rm v}(I(C_5))$. We prove that $C_5$ is the only cycle 
where the inequality ${\rm v}(I(C_s))\leq{\rm im}(C_s)$ fails.

\noindent \textbf{Theorem~\ref{cycles-indmat}.}\textit{
Let $C_s$ be an $s$-cycle and let $I(C_s)$ be its edge ideal. Then,
${\rm v}(I(C_s))\leq{\rm im}(C_s)$ if and only if $s \neq 5$.
}

If $v\in V(G)$, we denote the closed neighborhood of $v$ by $N_G[v]$.
A vertex $v$ of $G$ is called {\it simplicial} if the induced 
subgraph $H=G[N_{G}[v]]$ on the vertex set $N_G[v]$ is a complete graph. A subgraph $H$ of $G$ is
called a \textit{simplex} if $H=G[N_{G}[v]]$ for some simplicial
vertex $v$. A graph $G$ is \textit{simplicial} if every vertex of $G$ 
is either simplicial
or is adjacent to a simplicial vertex of $G$.      

If $A$ is a
stable set of a graph $G$, $H_i$ is a complete subgraph of $G$ for
$i=1,\ldots,r$ and $A\bigcup \{V(H_i)\}_{i=1}^r$ is a partition of
$V(G)$, then ${\rm reg}(S/I(G))\leq r$
\cite[Theorem~2]{woodroofe-matchings}. We consider a special type of
partitions of $V(G)$ that allow us to link $\mathcal{A}_G$ with induced
matchings of $G$. A graph $G$ has a \textit{simplicial partition} if $G$ has simplexes $H_1, \ldots, H_r$,
such that $\{V(H_i)\}_{i=1}^r$ is a partition of $V(G)$. 
Our next
result shows that ${\rm v}(I(G))\leq{\rm im}(G)$ if $G$ has a
simplicial partition.

\noindent \textbf{Theorem~\ref{Domi-InduceMatch-simplex}.}\textit{
Let $G$ be a graph with simplexes $H_1, \ldots, H_r$,
such that $\{V(H_i)\}_{i=1}^r$ is a partition of $V(G)$.
If $G$ has no isolated vertices, then there
is $D=\{y_1,\ldots,y_k\} \in {\mathcal A}_G$,  
and there are simplicial vertices $x_1,\ldots,x_k$ of $G$ and integers
$1\leq j_1< \cdots <j_k \leq r$ such that $P=\{\{x_i,y_i\}\}_{i=1}^k$ is an induced matching of $G$ and
 $H_{j_i}$ is the induced subgraph $G[N_G[x_i]]$ on $N_G[x_i]$ for $i=1,\ldots,k$. Furthermore ${\rm
 v}(I(G))\leq|D|=|P|\leq{\rm im}(G)\leq{\rm reg}(S/I(G))$. 
}

As a consequence, using a result of Finbow, Hartnell and Nowakowski 
that classifies the connected well-covered graphs without $4$- and
$5$-cycles \cite[Theorem 1.1]{Finbow2}
(Theorem~\ref{wellcovered-characterization1}), we show two more
families of graphs where the v-number is a lower bound for the
regularity.

\noindent \textbf{Corollary~\ref{simplicial-4-5-cycles}.}\textit{
Let $G$ be a well-covered graph and let $I(G)$ be its edge ideal. If $G$ is simplicial
or $G$ is connected and contain neither $4$- nor $5$-cycles, then 
$${\rm v}(I(G))\leq{\rm im}(G)
\leq{\rm reg}(S/I(G))\leq\beta_0(G).$$
}

A vertex $v$ of a graph $G$ is called a \textit{shedding vertex} if 
 each stable set of $G\setminus N_{G}[v]$ is not a maximal stable set
 of $G\setminus v$. We prove that every vertex of $G$ is a shedding
 vertex if and only if ${\mathcal A}_{G}={\mathcal F}_{G}$
 (Proposition~\ref{Shedding-stable}).

A graph $G$ belongs to class $W_2$ if
$|V(G)|\geq 2$ and any two disjoint stable sets $A_1,A_2$ are contained in two 
disjoint maximum stable sets $B_1, B_2$ with $|B_i|=\beta_0(G)$ for
$i=1,2$. 
A graph $G$ is in $W_2$ if and only if $G$ is well-covered, 
$G\setminus v$ is well-covered for all $v\in
V(G)$ and $G$ has no isolated vertices  \cite[Theorem~2.2]{Levit-Mandrescu}.  A graph $G$
without isolated vertices is in $W_2$ if and only if ${\rm
v}(I(G))=\beta_0(G)$ \cite[Theorem~4.5]{v-number}. As an
application we recover the only if implication of this result
(Corollary~\ref{w2-graph}). Using that a graph $G$ without isolated vertices is in
$W_2$ if and only if $G$ is well-covered and ${\mathcal
A}_{G}={\mathcal F}_{G}$ \cite[Theorem 4.3]{v-number}, by 
Proposition~\ref{Shedding-stable}, we recover the
fact that a graph $G$ without isolated vertices is in $W_2$ if and only if $G$ is well-covered and every
$v\in V(G)$ is a shedding vertex \cite[Theorem 3.9]{Levit-Mandrescu}. For other
characterizations of graphs in $W_2$ see
\cite{Levit-Mandrescu,Staples} and the references therein.

In Section~\ref{examples-section} we show examples illustrating some
of our results. In particular in Example~\ref{example2} we compute the
combinatorial and algebraic invariants of the well-covered graphs
$C_7$ and $T_{10}$ that are 
depicted in Figure~\ref{C7-T10}. These two graphs occur in the
classification of connected well-covered graphs without $4$- and
$5$-cycles \cite[Theorem 1.1]{Finbow2}
(Theorem~\ref{wellcovered-characterization1}). A related result is the
characterization of
well-covered graphs of girth at least $5$ given in \cite{finbow1}.

For all unexplained terminology and additional information,  we refer to 
\cite{diestel,Har} for the theory of graphs and
\cite{graphs-rings,Herzog-Hibi-book,monalg-rev} for the theory of
edge ideals and monomial ideals.

\section{Preliminaries}\label{prelim-section} 

In this section we give some definitions and present some well-known
results that will be used in the following sections.  
To avoid repetitions, we continue to employ 
the notations and
definitions used in Section~\ref{intro-section}.

\begin{definition}\cite{eisenbud-syzygies}\label{regularity-def} Let $I\subset S$ be a graded ideal and let
${\mathbf F}$ be the minimal graded free resolution of $S/I$ as an
$S$-module:
\[
{\mathbf F}:\ \ \ 0\rightarrow
\bigoplus_{j}S(-j)^{b_{g,j}}
\stackrel{}{\rightarrow} \cdots
\rightarrow\bigoplus_{j}
S(-j)^{b_{1,j}}\stackrel{}{\rightarrow} S
\rightarrow S/I \rightarrow 0.
\]
The {\it Castelnuovo--Mumford regularity\/} of $S/I$ ({\it
regularity} of $S/I$ for short) is defined as
$${\rm reg}(S/I):=\max\{j-i \mid b_{i,j}\neq 0\}.
$$
\quad The integer $g$, denoted ${\rm pd}(S/I)$, is the \textit{projective dimension} of $S/I$.  
\end{definition}

Let $G$ be a graph with vertex set $V(G)$. Given $A\subset V(G)$, 
the \textit{induced subgraph} on $A$, denoted $G[A]$, is the maximal
subgraph of $G$ with vertex set $A$. The edges of $G[A]$ are all the edges of $G$ that are
contained in $A$. The induced subgraph $G[V(G)\setminus A]$ of $G$ on the
vertex set $V(G)\setminus A$ is denoted by
$G\setminus A$.  
If $v$ is a vertex of $G$, then we denote the neighborhood of
$v$ by $N_G(v)$ and the closed neighborhood $N_G(v)\bigcup\{v\}$
of $v$ by $N_G[v]$. Recall that $N_G(v)$ is the set of all vertices 
of $G$ that are adjacent to $v$. If $A\subset V(G)$, we 
set $N_G(A):=\bigcup_{a\in A} N_G(a)$. 

\begin{theorem}\cite{Campbell}\label{Campbell-theo} 
If a graph $G$ is well-covered and is not complete, then
$G_v:=G\setminus N_G[v]$ is well-covered
for all $v$ in $V(G)$. Moreover, $\beta_0(G_v) = \beta_0(G) - 1$.
\end{theorem}

If $G$ is a graph, then $\beta_1(G)\leq \alpha_0(G)$. We say that $G$
is a \textit{K\H{o}nig graph} if $\beta_1(G)=\alpha_0(G)$. This notion
can be used to classify very well-covered graphs
(Proposition~\ref{konig-vwc}).

\begin{theorem}\label{lemma-Ivan}{\rm
(\cite[Theorem 5]{disc-math}, \cite[Lemma~2.3]{susan-reyes-vila})} 
Let $G$ be a graph without isolated vertices. If $G$ is a graph
without $3$-, $5$-, and $7$-cycles or $G$ is a K\H{o}nig graph, then $G$ is
well-covered if and only if $G$ is very well-covered. 
\end{theorem}

\begin{definition}\label{p-def1}
A perfect matching $P$ of a
graph $G$ 
is said to have property {\bf(P)} 
if for all $\{a,b\}$, $\{a^{\prime},b^{\prime}\}\in E(G)$, and
$\{b,b^{\prime}\}\in P$, one has $\{a,a^{\prime}\}\in E(G)$.
\end{definition}

\begin{remark}\label{sep26-19}
Let $P$ be a perfect matching of a graph $G$ with property {\bf(P)}.
Note that if $\{b,b'\}$ is in $P$ and $a\in V(G)$, then
$\{a,b\}$ and $\{a,b'\}$ cannot be both in $E(G)$ because $G$ 
has no loops. In other words $G$ has no triangle containing 
an edge in $P$.
\end{remark}

\begin{theorem}\label{konig}{\rm \cite[Theorem 1.2]{favaron}}
The following conditions are equivalent for a graph $G$: 
\begin{enumerate}
\item $G$ is very well-covered.
\item $G$ has a perfect matching with property {\bf{(P)}}. 
\item $G$ has a perfect matching, and each perfect matching of $G$ has
property {\bf{(P)}}.
\end{enumerate}
\end{theorem}

Let $G$ be a graph with vertex set $V(G)=\{t_1,\ldots,t_s\}$ and let 
$U=\{u_1,\ldots, u_s\}$ be a new set of vertices. The {\it whisker graph} or {\it
suspension\/} of $G$, denoted by
$W_G$, is the graph obtained from $G$ by attaching to each 
vertex $t_i$ a new vertex $u_i$ and a new edge $\{t_i,u_i\}$.  
The edge $\{t_i,u_i\}$ is called a {\it whisker} or \textit{pendant
edge}. The graph $W_G$ was 
introduced in \cite{ITG} as a device to study the numerical invariants and
properties of graphs and edge ideals.  

\begin{lemma}\label{bipartite-whiskers}
Let $G$ be a graph without isolated vertices. The following hold.
\begin{itemize}
\item[(a)] If $G$ is a bipartite well-covered graph, 
then $G$ is very well-covered. 
\item[(b)] The whisker graph $W_G$ of $G$ is very well-covered.
\end{itemize}
\end{lemma}

\begin{proof} (a) A bipartite well-covered graph without
isolated vertices has a perfect matching $P$ that satisfies 
property {\bf(P)} \cite[Theorem~1.1]{unmixed}. Thus, by
Theorem~\ref{konig}, $G$ is very well-covered. 

(b): The perfect matching $P=\{\{t_i,u_i\}\}_{i=1}^n$ of the whisker 
graph $W_G$ 
satisfies property {\bf(P)} and, by
Theorem~\ref{konig}, $G$ is very well-covered.
\end{proof}

%If $G$ is a graph, then $\beta_1(G)\leq \alpha_0(G)$. We say that $G$
%is a \textit{K\H{o}nig graph} if $\beta_1(G)=\alpha_0(G)$. This notion
%can be used to classify very well-covered graphs.

\begin{proposition}\cite[Lemma~17]{Ivan-Reyes}\label{konig-vwc} Let $G$ be a graph without
isolated vertices. Then, $G$ is a very well-covered graph if and only
if $G$ is well-covered and $\beta_1(G)=\alpha_0(G)$.
\end{proposition}

\begin{proof}
$\Rightarrow$) Assume that $G$ is very well-covered, then
$|V(G)|=2\alpha_0(G)$. It suffices to show that $\beta_1(G)=\alpha_0(G)$. In
general $\beta_1(G)\leq \alpha_0(G)$. By Theorem~\ref{konig}, $G$ has
a perfect matching $P=\{e_1,\ldots,e_r\}$. Then,
$|V(G)|=2r=2\alpha_0(G)$ and $r=\alpha_0(G)$. Thus,
$\alpha_0(G)=|P|\leq\beta_1(G)$, and one has $\alpha_0(G)=\beta_1(G)$.

$\Leftarrow$) Assume that $G$ is well-covered and
$\beta_1(G)=\alpha_0(G)$. Let $P=\{e_1,\ldots,e_r\}$ be a matching of $G$ with 
$r=\beta_1(G)$. We need only show that $|V(G)|=2\alpha_0(G)$. Clearly
$|V(G)|$ is greater than or equal 
to $2\alpha_0(G)$ because $\bigcup_{i=1}^re_i\subset V(G)$. We proceed by contradiction assuming that
$\bigcup_{i=1}^re_i\subsetneq V(G)$. Pick $v\in
V(G)\setminus\bigcup_{i=1}^re_i$. As $v$ is not an isolated vertex of
$G$, there is a minimal vertex cover $C$ of $G$ that contains $v$. As
$G$ is well-covered one has that $|C|=\alpha_0(G)=r$. Since $e_i\bigcap
C\neq\emptyset$ for $i=1,\ldots,r$ and $v\in C$, we get $|C|\geq r+1$,
a contradiction. 
\end{proof}

We say that a graph $G$ is in the family $\mathcal{F}$ if there 
exists $\{x_1,\ldots,x_k\}\subset V(G)$ where for each $i$, $x_i$ is
simplicial, $|N_G[x_i]|\leq 3$ and $\{N_G[x_i]\mid i=1,\ldots,k\}$ is
a partition of $V(G)$.

\begin{theorem}\label{wellcovered-characterization1}{\rm \cite[Theorem 1.1]{Finbow2}}
Let $G$ be a connected graph that contain neither $4$- nor $5$-cycles and let
$C_7$ and $T_{10}$ be the two graphs in Example~\ref{example2}. Then
$G$ is a well-covered graph if and only if 
$G\in \{C_7,T_{10}\}$ or $G\in\mathcal{F}$.
\end{theorem}

\begin{theorem}\label{lower-bound-reg} 
Let $G$ be a graph. The following hold.
\begin{enumerate}
\item [(a)] {\rm(\cite[Theorem 4.5]{Beyarslan-etal},
\cite{katzman1})} 
$2(n-1)+{\rm im}(G)\leq{\rm
reg}(S/I(G)^n) $ for all $n\geq 1$.
\item[(b)] {\rm(\cite[Theorem 4.7]{Beyarslan-etal},
\cite{JS-very-well-covered})} If $G$ is a forest or
$G$ is very well-covered, then 
$${\rm
reg}(S/I(G)^n)=2(n-1)+{\rm im}(G)\text{ for all } n\geq 1.
$$
\item[(c)] \cite[Theorem 1.3]{Mahmoudi-et-al} If $G$ is very
well-covered, then ${\rm reg}(S/I(G)) = {\rm im}(G)$.
\end{enumerate}
\end{theorem}

\section{The \text{v}-number of a graded ideal}\label{vnumber-section}
Let $S=K[t_1,\ldots,t_s]=\bigoplus_{d=0}^{\infty} S_d$ be a polynomial ring over
a field $K$ with the standard grading and let $I$ be a graded ideal of
$S$. In this section we show a formula for the v-number of $I$ that
can be used to compute this number using \textit{Macaulay}$2$
\cite{mac2}. 
To avoid repetitions, we continue to employ 
the notations and
definitions used in Sections~\ref{intro-section} and
\ref{prelim-section}.

\begin{lemma}\label{mingens-lemma}
Let $I\subset S$ be a graded ideal. If $(I\colon f)=\mathfrak{p}$ for
some prime ideal $\mathfrak{p}$ and some $f\in S_d$, $d\geq 0$, then
$I\subsetneq(I\colon\mathfrak{p})$ and there is a
minimal homogeneous generator $\overline{g}:=g+I$ of
$(I\colon\mathfrak{p})/I$ 
such that $\deg(f)\geq \deg(g)$ and $(I\colon g)=\mathfrak{p}$.
\end{lemma}

\begin{proof} The strict inclusion $I\subsetneq(I\colon\mathfrak{p})$
is clear because $f\in(I\colon\mathfrak{p})\setminus I$. Let
$\mathcal{G}=\{\overline{g}_1,\ldots,\overline{g}_r\}$ be a minimal 
generating set of $(I\colon\mathfrak{p})/I$ such that
$g_i$ is a homogeneous polynomial for all $i$. As
$(I\colon f)=\mathfrak{p}$, one has $\overline{f}\neq\overline{0}$ and
$f\in(I\colon\mathfrak{p})$. Then, we can choose homogeneous
polynomials $h_1,\ldots,h_r$ in $S$, $p$ in $I$, such that 
$f=\sum_{i=1}^rh_ig_i+p$ and $d=\deg(h_ig_i)$ for all $i$ with
$h_i\neq 0$. One
has the inclusion $\bigcap_{i=1}^r(I\colon g_ih_i)\subset(I\colon
f)$. Indeed, take $h$ in $\bigcap_{i=1}^r(I\colon g_ih_i)$, then
$hh_ig_i\in I$ for all $i$ and $hf=\sum_{i=1}^rhh_ig_i+hp$ is in $I$,
thus $h$ is in $(I\colon f)$.
Therefore, using that all $g_i$'s are in
$(I\colon\mathfrak{p})$, one has the inclusions
$$
\mathfrak{p}\ \mathlarger{\subset}\bigcap_{i=1}^r(I\colon g_i)\
\mathlarger{\subset}\bigcap_{i=1}^r(I\colon g_ih_i)\
\mathlarger{\subset}(I\colon f)=\mathfrak{p},
$$
and consequently $\mathfrak{p}=\bigcap_{i=1}^r(I\colon g_ih_i)$. Hence,
by \cite[p.~74, 2.1.48]{monalg-rev}, we get $(I\colon
g_ih_i)=\mathfrak{p}$ for some $1\leq i\leq r$. As $g_i$ is in
$(I\colon\mathfrak{p})$, we obtain 
$$
\mathfrak{p}\subset(I\colon g_i)\subset(I\colon g_ih_i)=\mathfrak{p}.
$$
\quad Hence $\mathfrak{p}=(I\colon g_i)$ and
$d=\deg(f)=\deg(g_ih_i)\geq\deg(g_i)$.
\end{proof}

\begin{theorem}\label{vnumber-general} 
Let $I\subset S$ be a graded ideal and let $\mathfrak{p}\in{\rm
Ass}(I)$. The following hold.

\begin{enumerate}
\item[(a)]
If $\mathcal{G}=\{\overline{g}_1,\ldots,\overline{g}_r\}$ is
a homogeneous minimal generating set of $(I\colon\mathfrak{p})/I$,
then  
$$
{\rm v}_{\mathfrak{p}}(I)=
\min\{\deg(g_i)\mid 1\leq i\leq r\mbox{ and }(I\colon
g_i)=\mathfrak{p}\}. 
$$
\item[(b)] ${\rm v}(I)=\min\{{\rm v}_{\mathfrak{q}}(I)\mid
\mathfrak{q}\in{\rm Ass}(I)\}$. 
\item[(c)] ${\rm
v}_{\mathfrak{p}}(I)\geq\alpha((I\colon\mathfrak{p})/I)$ with equality
if $\mathfrak{p}\in{\rm Max}(I)$.
\item[(d)] If $I$ has no embedded primes, then $
{\rm v}(I)=\min\{\alpha\left((I\colon\mathfrak{q})/{I}\right)\vert\,
\mathfrak{q}\in{\rm Ass}(I)\}.
$
\end{enumerate}
\end{theorem}

\begin{proof} (a): Take any homogeneous polynomial $f$ in $S$ such that $(I\colon
f)=\mathfrak{p}$. Then, by Lemma~\ref{mingens-lemma}, there is
$g_j\in\mathcal{G}$ such that $\deg(f)\geq \deg(g_j)$ and $(I\colon
g_j)=\mathfrak{p}$. Thus, the set $\{g_i\mid (I\colon
g_i)=\mathfrak{p}\}$ is not empty and the inequality 
$$
{\rm v}_{\mathfrak{p}}(I)\leq \min\{\deg(g_i)\mid 1\leq i\leq r\mbox{
and }(I\colon 
g_i)=\mathfrak{p}\} 
$$
follows by definition of ${\rm v}_{\mathfrak{p}}(I)$. Now, we can pick a homogeneous
polynomial $f$ in $S$ such
that $\deg(f)={\rm v}_\mathfrak{p}(I)$ and $(I\colon
f)=\mathfrak{p}$ . Then, by Lemma~\ref{mingens-lemma}, there is
$g_j\in\mathcal{G}$ such that $\deg(f)\geq \deg(g_j)$ and $(I\colon
g_j)=\mathfrak{p}$. Thus, $\deg(f)=\deg(g_j)$ and the inequality
``$\geq$'' holds. 

(b): This follows at once from the definitions of ${\rm v}(I)$ and
${\rm v}_{\mathfrak{q}}(I)$.

(c): Pick a homogeneous polynomial $g$ in $S$ such that $\deg(g)={\rm
v}_{\mathfrak{p}}(I)$ and $(I\colon g)=\mathfrak{p}$. Then, $g\notin
I$ and $g\mathfrak{p}\subset I$, that is, $g$ is in
$(I\colon\mathfrak{p})\setminus I$. Thus ${\rm
v}_{\mathfrak{p}}(I)\geq\alpha((I\colon\mathfrak{p})/I)$.  Now,
assume that $\mathfrak{p}$ is in ${\rm Max}(I)$. To show
the reverse inequality take any homogeneous polynomial $f$ in
$(I\colon\mathfrak{p})\setminus I$. Then $f\mathfrak{p}\subset I$ and
$\mathfrak{p}\subset(I\colon f)$. Since ${\rm Ass}(I\colon f)$ is
contained in ${\rm Ass}(I)$, there is $\mathfrak{q}\in{\rm Ass}(I)$
such that $\mathfrak{p}\subset (I\colon f)\subset\mathfrak{q}$. Hence,
$\mathfrak{p}=\mathfrak{q}$ and $\mathfrak{p}=(I\colon f)$. Thus
${\rm v}_{\mathfrak{p}}(I)\leq\deg(f)$ and 
${\rm v}_{\mathfrak{p}}(I)\leq\alpha((I\colon\mathfrak{p})/I)$. 

(d): This follows immediately from (b) and (c).
\end{proof}

We give a direct proof of the next result that in particular relates the v-number of
a Cohen--Macaulay monomial ideal $I\subset S$ with that of $(I,h)$,
where $h\in S_1$ and $(I\colon h)=I$.

\begin{corollary}\cite[Proposition~4.9]{v-number}
Let $I\subset S$ be a Cohen--Macaulay non-prime graded ideal whose associated primes are
generated by linear forms and let $h\in S_1$ be a regular element on
$S/I$. Then ${\rm v}(I,h)\leq{\rm v}(I)$.
\end{corollary}

\begin{proof} By Theorem~\ref{vnumber-general}, 
there are $\mathfrak{p}\in{\rm Ass}(I)$ and 
$f\in(I\colon\mathfrak{p})\setminus I$ such that $\overline{f}=f+I$ is
a minimal generator of $M_\mathfrak{p}=(I\colon\mathfrak{p})/I$ and
$\deg(f)={\rm v}(I)$. The
associated primes of $(I\colon f)$ are contained in ${\rm Ass}(I)$,
thus there is $\mathfrak{q}\in{\rm Ass}(I)$
such that $\mathfrak{p}\subset(I \colon f)\subset\mathfrak{q}$. Hence,
$\mathfrak{p}=\mathfrak{q}$ because $I$ has no embedded associated
primes, and one has the equality $(I\colon f)=\mathfrak{p}$. We claim
that $f$ is not in $(I,h)$. 
By contradiction assume that $f\in(I,h)$. Then we can write
$f=f_1+hf_2$, with $f_i$ a homogeneous polynomial for $i=1,2$, $f_1\in I$,
$f_2\in S$. Hence, one has 
$$
\mathfrak{p}=(I\colon f)=(I\colon hf_2)=(I\colon f_2).
$$
\quad Therefore $f_2\in(I\colon\mathfrak{p})\setminus I$ and
$\overline{f}=\overline{h}\,\overline{f_2}$, a contradiction because
$\overline{f}$ is a minimal generator of $M_\mathfrak{p}$. This proves that $f\notin(I,h)$. Next we show
the equality $(\mathfrak{p},h)=((I,h)\colon f)$. The inclusion
``$\subset$'' is clear because $(I\colon f)=\mathfrak{p}$. Take an
associated prime $\mathfrak{p}'$ of $((I,h)\colon f)$. The height of
$\mathfrak{p}'$ is equal to ${\rm ht}(I)+1$ because $(I,h)$ is
Cohen--Macaulay and the associated primes of $((I,h)\colon f)$ are
contained in ${\rm Ass}(I,h)$. Then 
$$
\mathfrak{p}=(I\colon f)\subset ((I,h)\colon f)\subset\mathfrak{p}',
$$
and consequently $(\mathfrak{p},h)\subset ((I,h)\colon
f)\subset\mathfrak{p}'$. Now, $(\mathfrak{p},h)$ is prime because
$\mathfrak{p}$ is generated by linear forms, and ${\rm
ht}(\mathfrak{p},h)={\rm ht}(\mathfrak{p})+1={\rm ht}(I)+1$ because
$I$ is Cohen--Macaulay and $h$ is a regular element on $S/I$. Thus, 
$(\mathfrak{p},h)=\mathfrak{p}'$, $(\mathfrak{p},h)=((I,h)\colon
f)$, and ${\rm v}(I,h)\leq {\rm v}(I)$.
\end{proof}

\begin{proposition}\label{reg-dim-pol} Let $I\subset S$ be a monomial ideal minimally generated
by $G(I)$ and for each $t_i$ that occurs
in a monomial of $G(I)$ let $\gamma_i:=\max\{\deg_{t_i}(g)\vert\, g\in G(I)\}$.
Then 
$$
{\rm reg}(S/I)\leq\dim(S/I)+\textstyle\sum_{i}(\gamma_i-1).
$$
\end{proposition}

\begin{proof} To show this inequality we use the polarization
technique due to Fr\"oberg 
(see \cite{depth-monomial} and \cite[p.~203]{monalg-rev}). 
To polarize $I$ we use the set of new variables
$$
T_I=\textstyle\bigcup_{i=1}^n\{t_{i,2},\ldots,
t_{i,\gamma_i}\},
$$
where $\{t_{i,2},\ldots,t_{i,\gamma_i}\}$ is empty if $\gamma_i=1$.
Note that $|T_I|=\sum_i(\gamma_i-1)$. We identify the variable $t_i$ with $t_{i,1}$ for all
$i$. A power $t_i^{c_i}$ of a variable $t_i$, $1\leq
c_i\leq\gamma_i$, polarizes 
to $(t_i^{c_i})^{\rm pol}=t_i$ if $\gamma_i=1$, to
$(t_i^{c_i})^{\rm pol}=t_{i,2}\cdots t_{i,c_i+1}$ if $c_i<\gamma_i$, and
to $(t_i^{c_i})^{\rm pol}=t_{i,2}\cdots t_{i,\gamma_i}t_i$ if $c_i=\gamma_i$.
Setting $G(I)=\{g_1,\ldots,g_r\}$, the polarization $I^{\rm pol}$ of $I$ is the ideal
of $S[T_I]$ generated by $g_1^{\rm pol},\ldots,g_r^{\rm pol}$.
According to \cite[Corollary 1.6.3]{Herzog-Hibi-book} one has 
$${\rm reg}(S/I)={\rm reg}(S[T_I]/I^{\rm pol})\mbox{ and }{\rm
ht}(I)={\rm ht}(I^{\rm pol}).$$
\quad As $I^{\rm pol}$ is squarefree, by \cite[Proposition 3.2]{v-number},
one has ${\rm reg}(S[T_I]/I^{\rm pol})\leq\dim(S[T_I]/I^{\rm pol})$.
Hence, we obtain
$$
{\rm reg}(S/I)={\rm reg}(S[T_I]/I^{\rm pol})\leq\dim(S[T_I]/I^{\rm
pol})=\dim(S[T_I])-{\rm
ht}(I).
$$
\quad To complete the proof notice that $\dim(S[T_I])-{\rm
ht}(I)=\dim(S/I)+|T_I|$.
\end{proof}

A result of Beintema \cite{beintema} shows that a zero-dimensional
monomial ideal is Gorenstein if and only if it is a complete
intersection. The next result classifies the complete intersection
property using the regularity.

\begin{proposition} Let $I$ be a monomial ideal of $S$ of dimension
zero minimally generated by
$G(I)=\{t_1^{d_1},\ldots,t_s^{d_s},t^{d_{s+1}},\ldots,t^{d_m}\}$, where
$d_i\geq 1$ for $i=1,\ldots,s$ and $d_i\in\mathbb{N}^s\setminus\{0\}$
for $i>s$. Then ${\rm reg}(S/I)\leq\sum_{i=1}^s(d_i-1)$, with 
equality if and only if $I$ is a complete intersection.
\end{proposition}

\begin{proof} The inequality ${\rm reg}(S/I)\leq\sum_{i=1}^s(d_i-1)$
follows directly from Proposition~\ref{reg-dim-pol} because $\dim(S/I)=0$. If
$I$ is a complete intersection, then $I=(t_1^{d_1},\ldots,t_s^{d_s})$
and, by \cite[Lemma~3.5]{Chardin}, we get ${\rm
reg}(S/I)=\sum_{i=1}^s(d_i-1)$. Conversely assume that ${\rm
reg}(S/I)$ is equal to $\sum_{i=1}^s(d_i-1)$. We proceed by
contradiction assuming that $m>s$. Then the exponents of the monomial
$t^{d_m}=t_1^{c_1}\cdots t_s^{c_s}$ satisfy $c_i\leq d_i-1$ for
$i=1,\ldots,s$ because $t^{d_m}\in G(I)$. The regularity of $S/I$ is
the largest integer $d\geq 0$ such that $(S/I)_d\neq(0)$
\cite[Proposition~4.14]{eisenbud-syzygies}. Pick a monomial
$t^a=t_1^{a_1}\cdots t_s^{a_s}$ such that $t^a\in S_d\setminus I$ and
$d=\sum_{i=1}^s(d_i-1)$. Then, $a_i\leq d_i-1$ for $i=1,\ldots,s$
because $t^a$ is not in $I$, and consequently $a_i=d_i-1$ for
$i=1,\ldots,s$. Hence, $t^a=t^\delta t^{d_m}$ for some
$\delta\in\mathbb{N}^s$, a contradiction.
\end{proof}

\section{Induced matchings and the \text{v}-number}
In this section we show that the 
induced matching number of a graph $G$ is an upper
bound for the v-number of $I(G)$ when $G$ is very well-covered, or $G$ 
has a simplicial partition, or $G$ is well-covered connected and
contain neither $4$- nor $5$-cycles. We classify when the upper bound holds when $G$ is a
cycle, and classify when all vertices of a graph $G$ are shedding
vertices, we use this to gain insight on $W_2$-graphs. To avoid repetitions, we continue to employ 
the notations and definitions used in Sections~\ref{intro-section} and
\ref{prelim-section}.

\begin{theorem}\label{v(I)}{\rm \cite[Theorem~3.5]{v-number}}
If $I = I(G)$ is the edge ideal of a graph $G$, then
$\mathcal{F}_G\subset\mathcal{A}_G$ and the $\mathrm{v}$-number of $I$ is
$$
\mathrm{v}(I) = \min\{\vert A \vert : A\in\mathcal{A}_G\}. 
$$
\end{theorem}

\begin{lemma}\label{stable}
Let $A$ be a stable set of a graph $G$. If $N_G(A)$ is a vertex cover
of $G$, then
$A \in {\mathcal A}_G$.
\end{lemma}
\begin{proof}
We take any $b \in N_G(A)$, then there is $e \in E(G)$ such that $e \subset A \bigcup \{b\}$. 
Furthermore, $N_G(A) \bigcap A = \emptyset$, since $A$ is a stable set
of $G$. 
Thus, 
$$
e \textstyle\bigcap N_G(A) \subset (A\textstyle \bigcup
\{b\})\textstyle \bigcap N_G(A) \subset \{b\},
$$
and consequently $e\bigcap(N_G(A)\setminus\{b\})=\emptyset$. 
Hence, $N_G(A) \setminus \{ b\}$ is not a vertex cover of $G$, since
$e \in E(G)$. Therefore $N_G(A)$ is a minimal vertex cover of $G$ and
$A \in {\mathcal A}_G$
\end{proof}

\begin{theorem}\label{Domi-InduceMatch}
Let $G$ be a very well-covered graph and let $P=\{e_1, \ldots, e_r\}$
be a perfect matching of $G$. 
Then, there is an induced submatching $P'$ of $P$ and $D \in {\mathcal A}_G$
such that $D \subset V(P')$ and $\vert e \bigcap D \vert = 1$ for
each $e\in P'$. 
Furthermore ${\rm v}(I(G))\leq|P'|=|D|\leq{\rm im}(G)\leq{\rm
reg}(S/I(G))$.
\end{theorem}

\begin{proof} To show the first part we use induction on $\vert P
\vert$. If $r=1$, we set $P'=P=\{e_1\}$ and $D=\{x_1\}$, where
$e_1=\{x_1,y_1\}$. Assume $r>1$.  We set $e_r=\{x,x'\}$, 
$G_1: = G\setminus\{x,x'\}$ and $P_1:= P \setminus \{e_r\}$.  
By Theorem~\ref{konig}, $P$ satisfies the property {\bf{(P)}}. 
Then, $P_1$ satisfies the property {\bf{(P)}} as well. Thus, by Theorem \ref{konig},
 $G_1$ is very well-covered with a perfect matching $P_1$. Hence, by induction hypothesis, there is
an induced submatching $P'_1$ of $P_1$ and $D_1 \in {\mathcal A}_{G_1}$ such that  $D_1 \subset  V(P'_1)$
and $\vert e \bigcap D_1 \vert = 1$ for each $e \in P'_1$. Consequently, $N_{G_1}(D_1)$ is a minimal vertex
cover of $G_1$. We will consider
two cases: $e_r \bigcap N_G(D_1) \neq \emptyset$ and $e_r \bigcap N_G(D_1) = \emptyset$.

Case (I): Assume that $e_r \bigcap N_G(D_1) \neq \emptyset$. Thus, we
may assume that there is $\{x,d\} \in E(G)$
with $d \in D_1$. Then, $N_G(x')\subset N_G(d) \subset N_G(D_1)$, since $P$ satisfies
property  {\bf{(P)}}. Hence, $N_G(D_1)$ is a vertex cover of $G$, since $N_{G_1}(D_1)$ is a vertex cover of
$G_1$ and $\{x\} \subset N_G(x') \subset N_G(D_1)$. Therefore, by
Lemma~\ref{stable},
$D_1 \in {\mathcal A}_G$, so this case follow by making $D=D_1$ and $P'=P'_1$.

Case (II): Assume that $e_r \bigcap N_G(D_1) = \emptyset$. We set
$D_2:=V(P'_1) \setminus D_1$, then $D_2$ 
is a stable set of $G_1$ and also of $G$, since $P'_1$ is an induced matching
of $G_1$ and also of $G$. One has the
inclusion
\begin{equation}\label{sep10-21-1}
V(P'_1)\textstyle \bigcap (N_G(x)\textstyle\bigcup N_G(x')) \subset
D_2,
\end{equation}
indeed take $z\in V(P'_1)\bigcap N_G(x)$ (the case $z\in
V(P'_1)\bigcap N_G(x')$ is similar). If $z\notin D_2$, then $z\in
D_1\bigcap N_G(x)$, $\{z,x\}\in E(G)$, and $x\in e_r\bigcap N_G(D_1)$, a
contradiction. We claim that $\vert e_r \bigcap N_G(D_2) \vert \leq
1$. By contradiction suppose 
$x,x' \in N_G(D_2)$, then there are $d_1,d_2 \in D_2$ such that 
$\{x,d_1\},\{x',d_2\} \in E(G)$. Then $\{d_1,d_2\} \in E(G)$, since $P$ satisfies
property {\bf{(P)}},  a contradiction, since $D_2$ is a stable set of
$G$. 
Hence, $\vert e_r \bigcap N_G(D_2) \vert \leq 1$ and we may assume 
\begin{equation}\label{sep10-21-2}
e_r \textstyle\bigcap N_G(D_2) \subset \{x\}.
\end{equation}
\quad Next we show that $V(P_1')\bigcap N_G(x')=\emptyset$. If the
intersection is non-empty, by Eq.~\eqref{sep10-21-1}, we can
pick $z$ in $D_2\bigcap N_G(x')$, then $\{z,x'\}\in E(G)$ and $x'\in
N_G(D_2)$, a contradiction to Eq.~\eqref{sep10-21-2}. Therefore, by
Eq.~\eqref{sep10-21-1}, we obtain the inclusion
$$  
V(P'_1)\textstyle \bigcap (N_G(x)\textstyle\bigcup N_G(x')) \subset
D_2 \textstyle\bigcap N_G(x) =:A_2, 
$$
\quad Thus, the edge set $Q:=\{e \in P'_1  \mid  e \bigcap A_2 =
\emptyset \} \bigcup \{e_r\}$, is 
an induced matching, since  $P'_1$ is an induced matching.
Setting 
\begin{equation}\label{oct2-21}
D_3:=\{ y \in D_1\mid \{y,y'\} \in P'_1 \mbox{ with } y'\notin A_2\}\textstyle \bigcup \{x\},
\end{equation}
i.e., $D_3=(D_1 \bigcap V(Q)) \bigcup \{x\}$, we get $\vert f \bigcap D_3 \vert = 1$ for each $f\in Q$,
since  $\vert e \bigcap D_1 \vert = 1$ for each $e \in P'_1$.
Note that $D_3$ is a stable set of $G$, since $D_1$ is a stable set and 
$\{x\} \bigcap N_G(D_1) = \emptyset$.
Now, take $ e\in E(G)$. We will prove that $e \bigcap N_G(D_3) \neq
\emptyset$. Clearly $N_G(x)\subset N_G(D_3)$ because $x\in D_3$. 
If $x' \in e$, then $x' \in e \bigcap N_G(x)  \subset e \bigcap  N_G(D_3)$. Now, if
$x \in e$, then $e=\{x,y\}$ for some $y$ in $V(G)$, and $y \in e \bigcap  N_G(x)  \subset e \bigcap N_G(D_3)$.
So, we may assume $e \bigcap \{x,x'\}=\emptyset$, then $e \in E(G_1)$.
Thus, there is $z \in e \bigcap N_{G_1}(D_1)$, since   $N_{G_1}(D_1)$
is a vertex cover of $G_1$. Then, there is $d \in D_1$, such that
$z \in N_{G_1}(d)$. If $d \in D_3$, then $z \in N_G(D_3) \bigcap e$.
Finally, if $d \notin D_3$, then by Eq.~\eqref{oct2-21} and the inclusion
$D_1\subset V(P_1')$, there is $d' \in A_2$ such that $\{d,d'\} \in P'_1$.
So, $\{x,d'\} \in E(G)$, since $d' \in A_2$ .
This implies, $ \{x,z\} \in E(G)$, since $\{d,z\}\in E(G)$,
$\{x,d'\}\in E(G)$, $\{d,d'\} \in P$, and 
$P$ satisfies property {\bf{(P)}}. Thus,  $z \in e \bigcap N_G(x)  \subset e \bigcap N_G(D_3)$.
Hence, $N_G(D_3)$ is a vertex cover and, by Lemma~\ref{stable}, $D_3 \in {\mathcal A}_G$.
Therefore, this case follows by making $P'=Q$ and $D=D_3$. This
completes the induction process. 

Next we show the equality $|P'|=|D|$. By the first part, we may assume that
$P'=\{e_1,\ldots,e_\ell\}$, $1\leq \ell\leq r$, $e_i=\{x_i,y_i\}$ for
$i=1,\ldots,\ell$, and $x_1,\ldots,x_\ell\in D$. Thus, $\ell=|P'|\leq
|D|$ and, since $D\subset V(P')$, we get $2|D|\leq 2|P'|$. 
Then $|P'|=|D|$. The inequality ${\rm v}(I(G))\leq |D|$ follows by Theorem~\ref{v(I)} 
and $|P'|\leq{\rm im}(G)$ is clear by definition of ${\rm
im}(G)$. Finally, the inequality ${\rm im}(G)\leq{\rm
reg}(S/I(G))$ follows directly from Theorem~\ref{lower-bound-reg}. 
\end{proof}

\begin{corollary}\cite[Theorem 3.19(b)]{v-number}\label{sep29-21} Let $G$ be a graph and let $W_G$ be
its whisker 
graph. Then 
$$
{\rm v}(I(W_G))\leq{\rm reg}(K[V(W_G)]/I(W_G)).
$$
\end{corollary}

\begin{proof} By Lemma~\ref{bipartite-whiskers}, $W_G$ is very
well-covered. Thus, by Theorem~\ref{Domi-InduceMatch}, the v-number of
$I(W_G)$ is bounded from above by the regularity of 
$K[V(W_G)]/I(W_G)$.
\end{proof}

\begin{lemma}\label{bound}
Let $\ell\geq 0$ and $s=4\ell+r$ be integers with $r \in \{0,1,2,3\}$.
If $s \geq 3$ and $s \neq 5$, then  
$$
\left\lfloor \frac{s} {3} \right\rfloor \geq  \ell \mbox { if }  r=0  \mbox{ and  } \left\lfloor \frac{s}{3} \right\rfloor \geq \ell+1 \mbox{ otherwise}.
$$
\end{lemma}

\begin{proof}
By the division algorithm, $ s\equiv r' \pmod{3}$, where $r' \in \{0,1,2\}$. Then 
$$
\left\lfloor \frac{s} {3} \right\rfloor= \frac{4\ell+r-r'} {3}=\ell+ \frac{\ell+r-r'} {3} \in \mathbb{Z}.
$$
\quad Thus, $a:= \frac{\ell+r-r'} {3} \in \mathbb{Z}$. If $r=0$, then
$a \geq 0$. This follows using that $0\leq
r'\leq 2$ and $\ell\geq 0$. Hence,  $ \left\lfloor \frac{s} {3}
\right\rfloor \geq  \ell $. Now, assume 
$r \in \{1, 2, 3\}$. We claim that $a \geq 1$. By contradiction assume
that $a \leq 0$, then  $\ell+r \leq r'$. If $\ell=0$, then $s=r=3$, 
since $s \geq 3$. A contradiction, since $3=\ell+r \leq r'$  and $r' \leq 2$. 
Thus, $\ell \geq 1$ and we have $2 \leq \ell+1 \leq \ell +r \leq r' \leq 2$. This implies
$\ell=1=r$ and $r'=2$. Consequently $s=5$, a contradiction. 
Therefore,  $a \geq 1$ and $\left\lfloor \frac{s}{3} \right\rfloor \geq \ell+1$. 
\end{proof}

\begin{theorem}\label{cycles-indmat}
Let $C_s$ be an $s$-cycle and let $I(C_s)$ be its edge ideal. Then,
${\rm v}(I(C_s))\leq{\rm im}(C_s)$ if and only if $s \neq 5$.
\end{theorem}

\begin{proof}
$\Rightarrow$) Assume that ${\rm v}(I(C_s))\leq{\rm im}(C_s)$. If $s=5$,
then ${\rm v}(I(C_s))=2$ and ${\rm im}(C_s)=1$, a contradiction. Thus
$s\neq 5$.

$\Leftarrow$) Assume that $s\neq 5$. We can write
$C_s=(t_1,e_1,t_2,\ldots, t_i,e_i,t_{i+1}, \ldots, t_s,e_s,t_1)$. The
matching $P=\{e_1,e_4, \ldots, e_{3q-2}\}$, where $q:=\lfloor \frac{s}{3} \rfloor$, 
is an induced matching of $C_s$ and $\vert P \vert=q$. 
Now we choose a stable set $A$ of $C_s$, for each one of the following cases:

Case $s=4\ell$. If $A=\{t_2, t_6, \ldots, t_{4\ell-2}\}$, then 
$N_{C_s}(A)=\{t_1,t_3,t_5,t_7, \ldots, t_{s-3},t_{s-1}\}$ is a vertex
cover of $G$ and  $\vert A \vert =\ell$.

Case $s=4\ell+1$. If $A=\{t_2, t_6, \ldots, t_{4\ell-2}\} \bigcup \{t_{4\ell}\}$, then 
$N_{C_s}(A)=\{t_1,t_3, \ldots, t_{s-4},t_{s-2}\} \bigcup \{t_s\}$ is a
vertex cover of $G$ and
$\vert A \vert =\ell+1$.

Case $s=4\ell+2$. If $A=\{t_2, t_6, \ldots, t_{4\ell+2}\}$, then 
$N_{C_s}(A)=\{t_1,t_3,t_5,t_7, \ldots, t_{s-3},t_{s-1}\}$ is a vertex
cover of $G$ and
$\vert A \vert =\ell+1$.

Case $s=4\ell+3$. If $A=\{t_2, t_6, \ldots, t_{4\ell+2}\}$, then 
$N_{C_s}(A)=\{t_1,t_3,t_5,t_7, \ldots, t_{s-2},t_s\}$ is a vertex
cover of $G$ and
$\vert A \vert =\ell+1$.

In each case $N_{C_s}(A)=\{t_i \vert i \mbox{ is odd} \}$ and 
$N_{C_s}(A)$ is a vertex cover of $G$. 
So, by Lemma~\ref{stable}, $A \in {\mathcal A}_{C_s}$. Now, assume 
$s=4\ell+r$, with $r \in \{0,1,2,3\}$ and $\ell\geq 0$ an integer.
Then, by Lemma~\ref{bound},  
$\lfloor \frac{s} {3} \rfloor \geq  \ell$
 if $ r=0$ and $\lfloor \frac{s}{3} \rfloor \geq \ell+1$ otherwise. 
Hence, $\vert  P \vert = \lfloor \frac{s}{3} \rfloor   \geq \vert A \vert$. 
Therefore, ${\rm im}(C_s) \geq
{\rm v}(I(C_s))$, since ${\rm im}(C_s) \geq \vert P \vert$ 
and  $\vert A \vert \geq{\rm v}(I(C_s))$. 
\end{proof}

\begin{remark} The induced matching number of the cycle $C_s$ is equal to 
$\lfloor\frac{s}{\scriptstyle 3} \rfloor$. The regularity of
$S/I(C_s)$ is equal to $\lfloor (s+1)/3\rfloor$ 
\cite[Proposition~10]{woodroofe-matchings}.
\end{remark}

\begin{lemma}\label{induced-lemma}
Let $G$ be a graph without isolated vertices and let $z_1,\ldots,z_m$
be vertices of $G$ such that $\{N_G[z_i]\}_{i=1}^m$ is a partition of
$V(G)$. If $G_1=G\setminus N_G[z_m]$, then
\begin{enumerate}
\item[(i)] $N_{G_1}[z_i]=N_G[z_i]$ for $i<m$, and
\item[(ii)] $G_1[N_{G_1}[z_i]]=G[N_G[z_i]]$ for $i<m$.
\end{enumerate}
\end{lemma}

\begin{proof} (i): Assume that $1\leq i\leq m-1$. Clearly
$N_{G_1}[z_i]\subset N_G[z_i]$ because $G_1$ is a subgraph of $G$. To
show the inclusion ``$\supset$'' take $z\in N_G[z_i]$. Then, $z=z_i$
or $\{z,z_i\}\in E(G)$. If $z\in
N_G[z_m]$, then $z\in N_G[z_m]\bigcap N_G[z_i]$,
a contradiction. Thus, $z\notin N_G[z_m]$ and, since $G_1$ is an
induced subgraph of $G$, we get $z=z_i$ or $\{z,z_i\}\in E(G_1)$.
Thus, $z\in N_{G_1}[z_i]$.

(ii): By part (i), one has
$N_G[z_i]=N_{G_1}[z_i]\subset V(G)\setminus N_G[z_m]=V(G_1)$. Then 
\begin{align*}
E(G[N_G[z_i]])&=\{e\in E(G)\mid e\subset N_G[z_i]\}=\{e\in E(G)\mid
e\subset N_{G_1}[z_i]\} \\
&=\{e\in E(G_1)\mid e\subset N_{G_1}[z_i]\}=E(G_1[N_{G_1}[z_i]]).
\end{align*}
\quad Thus, $E(G[N_G[z_i]])=E(G_1[N_{G_1}[z_i]])$.
\end{proof}

\begin{theorem}\label{Domi-InduceMatch-simplex}
Let $G$ be a graph with simplexes $H_1, \ldots, H_r$,
such that $\{V(H_i)\}_{i=1}^r$ is a partition of $V(G)$.
If $G$ has no isolated vertices, then there
is $D=\{y_1,\ldots,y_k\} \in {\mathcal A}_G$,  
and there are simplicial vertices $x_1,\ldots,x_k$ of $G$ and integers
$1\leq j_1< \cdots <j_k \leq r$ such that $P=\{\{x_i,y_i\}\}_{i=1}^k$ is an induced matching of $G$ and
 $H_{j_i}$ is the induced subgraph $G[N_G[x_i]]$ on $N_G[x_i]$ for $i=1,\ldots,k$. Furthermore ${\rm
 v}(I(G))\leq|D|=|P|\leq{\rm im}(G)\leq{\rm reg}(S/I(G))$. 
\end{theorem}

\begin{proof} We proceed by induction on $r$. If $r=1$, then $V(H_1)=V(G)$ and there is a
simplicial vertex $x_1$ of $G$ such that $H_1=G[N_G[x_1]]$ is a
complete graph with at least two vertices. Picking $y_1\in
N_G[x_1]$, $y_1\neq x_1$, one has $\{x_1\}\in\mathcal{A}_G$ and
$\{x_1,y_1\}$ is an induced matching. Now assume that $r>1$. 
We set $G_1: = G \setminus V(H_r)$. Note that
$H_1,\ldots,H_{r-1}$ are simplexes of $G_1$
(Lemma~\ref{induced-lemma}) and $\{V(H_i)\}_{i=1}^{r-1}$ is a
partition of $V(G_1)$. Then, by induction hypothesis, there is  
$D_1 =\{y_1,\ldots, y_{k'}\} \in {\mathcal A}_{G_1}$, and there are
simplicial vertices $x_1,\ldots,x_{k'}$ of $G_1$ and integers 
$1\leq j_1<\cdots <j_{k'} \leq r-1$, such that 
$P_1=\{\{x_1,y_1\}, \ldots, \{x_{k'},y_{k'}\}\}$ 
is an induced matching of $G_1$ and
$H_{j_i}=G_1[N_{G_1}[x_i]]$ for $i=1,\ldots,k'$. By 
Lemma~\ref{induced-lemma}, one has
$G_1[N_{G_1}[x_i]]=G[N_G[x_i]]$ for $i=1,\ldots,k'$. We can
write $H_r=G[N_G[x]]$ for some simplicial vertex $x$ of $G$. 

Case (I): Assume that $V(H_r) \setminus \{x\} \subset N_G(D_1)$. Then, $N_G(D_1)$ is a
vertex cover of $G$. Indeed, take any edge $e$ of $G$. If $e\bigcap V(H_r)=\emptyset$, then
$e$ is an edge of $G_1$ and is covered by $N_{G_1}(D_1)$. 
%Assume that
%$e\bigcap V(H_r)\neq\emptyset$. If $e\subset V(H_r)$, then there is
%$z\in e\setminus\{x\}$, $z\in V(H_r)\setminus\{x\}\subset N_G(D_1)$.
%If $e\not\subset V(H_r)$ and $x\notin e$, then there is $z\in e$ with
%$z\in V(H_r)\setminus\{x\}\subset N_G(D_1)$. If $e\not\subset V(H_r)$
%and $x\in e$, then $e=\{x,z\}$ with $z\in V(H_j)$ for some $j<r$
%because $\{V(H_i)\}_{i=1}^r$ is a partition of $V(G)$. Thus, 
%$z\in N_G[x]=V(H_r)$, and consequently $z\in V(H_j)\bigcap V(H_r)$, a
%contradiction. 
Assume that $e\cap V(H_r) \neq \emptyset$. If $x \notin e$, then there
is $z \in e$ with $z \in V(H_r)\setminus\{x\} \subset N_G(D_1)$. Now, if $x \in e$
then $e=\{x,z\}$ with $z \in N_G[x]\setminus\{x\}=V(H_r)\setminus\{x\} \subset N_G(D_1)$. 
This proves that $N_G(D_1)$ is a vertex cover of
$G$. Hence, 
by Lemma~\ref{stable}, $D_1 \in  {\mathcal A}_G$ and, noticing that
$P_1$ is an induced matching of $G$, this case follows by making
$D=D_1$ and $P=P_1$. 

Case (II): Assume that there is $y\in V(H_r) \setminus \{x\}$  such that $y \notin N_G(D_1)$. 
Then, $D_2:=D_1 \bigcup \{y\}$ is a stable set of $G$. Also, $ N_G(D_2)$ is a vertex cover of $G$,
since $N_{G_1}(D_1)$ is a vertex cover of $G_1$, $H_r$ is a complete
subgraph of $G$, and $V(H_r) \subset N_G[y]$.
Thus, by Lemma~\ref{stable}, $D_2$ is in ${\mathcal A}_G$. We set $x_{k'+1}:=x$, 
$y_{k'+1}:=y$ and $H_{j_{k'+1}}:=H_r$. Then,  $\{x_{k'+1}, y_{k'+1}\} \in  E(H_r)$ and 
$P_2:=P_1 \bigcup \{\{x_{k'+1}, y_{k'+1}\} \}$ is an induced matching of $G$, 
since $P_1$ is an induced matching of $G_1$, $y \in V(H_r)\setminus N_G(D_1)$
and $H_{j_i}=G[N_G[x_i]]$, for $i=1, \ldots, k'+1$ . 
Therefore, this case follows by making $D=D_2$ and $P=P_2$.

The equality $|D|=|P|$ is clear.
The inequality ${\rm v}(I(G))\leq |D|$ follows from Theorem~\ref{v(I)} 
and $|P|\leq{\rm im}(G)$ is clear by definition of ${\rm
im}(G)$. Finally, the inequality ${\rm im}(G)\leq{\rm
reg}(S/I(G))$ follows directly from Theorem~\ref{lower-bound-reg}. 
\end{proof}

\begin{corollary}\label{simplicial-4-5-cycles}
Let $G$ be a well-covered graph and let $I(G)$ be its edge ideal. If $G$ is simplicial
or $G$ is connected and contain neither $4$- nor $5$-cycles, then 
$${\rm v}(I(G))\leq{\rm im}(G)
\leq{\rm reg}(S/I(G))\leq\beta_0(G).$$
\end{corollary}

\begin{proof} Assume that $G$ is simplicial. Let $\{z_1,\ldots,z_\ell\}$ be
the set of all simplicial vertices of $G$. Then $V(G)=\bigcup_{i=1}^\ell
N_G[z_i]$. As $G$ is well-covered, by \cite[Lemma~2.4]{Finbow2}, for
$1\leq i<j\leq \ell$ either $N_G[z_i]=N_G[z_j]$ or 
$N_G[z_i]\bigcap N_G[z_j]=\emptyset$. Thus there are simplicial
vertices $x_1,\ldots,x_k$ of $G$ such that $\{N_G[x_i]\}_{i=1}^k$ is a
partition of $V(G)$. Setting $H_i=G[N_G[x_i]]$ for $i=1,\ldots,k$ and
applying Theorem~\ref{Domi-InduceMatch-simplex}, we get that ${\rm
v}(I(G))\leq{\rm im}(G)\leq{\rm reg}(S/I(G))$.  Noticing that
$\dim(S/I(G))=\beta_0(G)$, the inequality ${\rm
reg}(S/I(G))\leq\beta_0(G)$ follows from
Proposition~\ref{reg-dim-pol}. 

Next assume that $G$ is connected and contain neither $4$- nor $5$-cycles. Then, by 
Theorem~\ref{wellcovered-characterization1}, 
$G\in \{C_7,T_{10}\}$ or $G\in\mathcal{F}$.
The cases $G = C_7$ or $G=T_{10}$ are treated in
Example~\ref{example2} (cf. Theorem~\ref{cycles-indmat}). 
If $G\in\mathcal{F}$, then there exists
$\{x_1,\ldots,x_k\}\subset V(G)$ where 
for each $i$, $x_i$ is
simplicial, $|N_G[x_i]|\leq 3$ and $\{N_G[x_i]\mid i=1,\ldots,k\}$ is
a partition of $V(G)$. In particular $G$ is simplicial and the
asserted inequalities follow from the first part of the proof.
\end{proof}

\begin{proposition}\label{Shedding-stable}
Let $G$ be a graph. The following conditions are equivalent.
\begin{enumerate}
\item Every vertex of $G$ is a shedding vertex.
\item ${\mathcal A}_{G}={\mathcal F}_{G}$.
\end{enumerate}
\end{proposition}

\begin{proof}
(1) $\Rightarrow$ (2): The inclusion ${\mathcal
A}_{G}\supset{\mathcal F}_{G}$ follows from Theorem~\ref{v(I)}. To
show the inclusion ${\mathcal
A}_{G}\subset{\mathcal F}_{G}$ we proceed by contradiction, suppose
there is $D \in {\mathcal A}_{G} \setminus {\mathcal F}_{G}$. 
 Then, $D$ is a stable set of $G$ and $N_G(D)$ is a vertex cover of $G$. 
Thus,  $D \bigcap N_G(D) = \emptyset$. Furthermore, since $D \notin {\mathcal F}_{G}$,
there is $x \in V(G) \setminus D$ such that $D \bigcup \{x\}$ is a stable
set of $G$. Then, $x \notin N_G(D)$. But $N_G(D)$ is a vertex cover
of $G$,  then 
$N_G(x) \subset N_G(D)$ and $A:=V(G)\setminus N_G(D)$ is a stable set
of $G$.
So, $A \subset V(G)\setminus N_G(x)$ and $A':=A \setminus x$  is a stable set of $V(G)\setminus N_G[x]$.
Now, we prove  $A'$ is a maximal stable set of $G \setminus x$.
By contradiction assume there is $a \in V(G \setminus x) \setminus A'$, such that 
$A' \bigcup \{a\}$ is a stable set. Then, $a \in N_G(D)$, since  $V(G)=A  \bigcup N_G(D)$. 
Also, $D\subset A'$, since  $D \bigcap N_G(D) = \emptyset$ and
$x\notin D$, a contradiction, since $a \in N_G(D)$ and $A' \bigcup \{a\}$ is a stable set. 
 Hence, $A'$ is a maximal stable set of  $G\setminus x$. 
Therefore $x$ is not a shedding vertex of $G$, a contradiction. 

(2) $\Rightarrow$ (1): By contradiction, suppose there is $x \in V(G)$
such that $x$ is not a shedding vertex. 
Thus, there is a maximal stable set $A$ of $G \setminus x$ such that  $A \subset V(G) \setminus N_G[x]$ .
Then, $C:=V(G\setminus x)\setminus A$  is a minimal vertex cover of $G\setminus x$
and $A \bigcup \{x\}$ is a stable set of $G$. So, $A \notin {\mathcal F}_G$.
Since $C$ is a minimal vertex cover of $G\setminus x$, we have that for each $z \in C$, there is
$z' \in V(G \setminus x) \setminus C = A$ such that $\{z,z'\} \in E(G)$. Consequently,
$C \subset N_G(A)$. Furthermore, if $a \in N_G(x)$, then $a \in G \setminus x$
and $a \notin A$. Thus, $a \in N_G(A)$, since $A$ is a maximal stable set of $G \setminus x$.
Hence, $N_G(x) \subset  N_G(A)$. This implies, $N_G(A)$ is a vertex cover of $G$, since
$C \subset N_G(A)$. Therefore, by Lemma~\ref{stable}, 
$A \in {\mathcal A}_G$, a contradiction since $A \notin {\mathcal F}_G$.
\end{proof}

\begin{lemma}\cite[cf. Corollary
3.3]{Levit-Mandrescu}\label{shedding-w2} If $G\in W_2$, then
every $v\in V(G)$ is a shedding vertex.
\end{lemma}

\begin{proof} Let $v$ be a vertex of $G$. We may assume that $G$ is
not a complete graph. Let $A$ be a stable set of 
$G_v:=G\setminus N_G[v]$. We proceed by contradiction assuming that
$A$ is a maximal stable set of $G\setminus v$. Then, as $G$ and
$G\setminus v$ are well-covered, we get  
$$\beta_0(G)=\beta_0(G\setminus v)=|A|.$$
\quad According to \cite[Theorem~5]{Pinter-jgt}, the graph $G_v$ is in $W_2$ and
 $\beta_0(G_v)=\beta_0(G)-1$. In particular $G_v$ is well-covered and
 $\beta_0(G_v)=\beta_0(G)-1$ (cf. Theorem~\ref{Campbell-theo}).  
%Hence, noticing that $A$ is also a
% maximal stable set of $G_v$, we obtain $\beta_0(G_v)=|A|$, a
% contradiction.
But $A$ is a stable set of $G_v$ and $\vert A \vert =\beta_0(G)$, a contradiction.
\end{proof}

\begin{corollary}\cite[Theorem~4.5]{v-number}\label{w2-graph}
If $G$ is a $W_2$-graph and $I=I(G)$, then ${\rm v}(I)=\beta_{0}(G)$.
\end{corollary}

\begin{proof}
By Theorem \ref{v(I)}, there is $D \in {\mathcal A}_G$ such that 
${\rm v}(I)=\vert D \vert$.
Since $G$ is a  $W_2$-graph, by Lemma~\ref{shedding-w2}, every vertex of $G$ is
a shedding vertex.
Thus,  by Proposition \ref{Shedding-stable}, $D \in {\mathcal F}_G$,
i.e., $D$ is a maximal 
stable set of $G$. Furthermore, $G$ is well-covered, since $G$ is a $W_2$-graph. Hence,
$\vert D \vert=\beta_{0}(G)$. Therefore, ${\rm v}(I)=\beta_{0}(G)$.
\end{proof}

\section{Examples}\label{examples-section}
\begin{example}\label{example1}
Let $S=\mathbb{Q}[t_1,t_2,t_3]$ be a polynomial ring and 
$I=(t_1^5,t_2^5,t_2^4t_3^5,t_1^4t_3^5)$. Then an irredundant 
primary decomposition of $I$ is given by
$$
I=(t_1^4,t_2^4)\textstyle\bigcap(t_1^5,t_2^5,t_3^5).
$$
\quad The associated primes of $I$
are $\mathfrak{p}_1=(t_1,t_2)$ and $\mathfrak{p}_2=(t_1,t_2,t_3)$.
Setting $g_1=t_1^4t_2^4$, $g_2=t_1^3t_2^3t_3^5$, and
$g_3=t_1^4t_2^4t_3^4$, and using Procedure~\ref{procedure1}, we get that $(I\colon\mathfrak{p}_1)/I$ and 
$(I\colon\mathfrak{p}_2)/I$ are minimally generated by
$\{\overline{g}_1,\overline{g}_2\}$
and $\{\overline{g}_3\}$, respectively. Using
Theorem~\ref{vnumber-general} and the equalities 
$$
(I\colon g_1)=(t_1,t_2,t_3^5),\ (I\colon g_2)=\mathfrak{p}_1,\ 
(I\colon g_3)=\mathfrak{p}_2,  
$$
we obtain that ${\rm v}(I)=11$. The regularity of the quotient ring
$S/I$ is equal to $12$. 
\end{example}

\begin{example}\label{example-wc}
Let $S=\mathbb{Q}[t_1,\ldots,t_6]$ be a polynomial ring and let $I$ be
the edge ideal
$$
I=(t_1t_2,\, t_2t_3,\, t_3t_4,\, t_1t_4,\, t_1t_5,\, t_2t_5,\,
t_3t_5,\, t_4t_5,\, t_1t_6,\, t_2t_6,\, t_3t_6,\, t_4t_6).
$$
\quad The graph $G$ defined by the generators of this ideal is
well-covered and not very well-covered, $\alpha_0(G)=4$, and ${\rm v}(I)={\rm im}(G)={\rm
reg}(S/I)=1$.
\end{example}

\begin{example}\label{example2}
Let $C_7$ and $T_{10}$ be the well-covered graphs of
Figure~\ref{C7-T10}. Let $R$ and $S$ be polynomial rings over the
field $\mathbb{Q}$ in the
variables $\{t_1,\ldots,t_7\}$ and $\{t_1,\ldots,t_{10}\}$,
respectively. 
Using \textit{Macaulay}$2$ \cite{mac2} and
Procedure~\ref{procedure1} we obtain ${\rm ht}(I(C_7))=\alpha_0(C_7)=4$, ${\rm
pd}(R/I(C_7))=5$ and 
$$
{\rm v}(I(C_7))=2={\rm im}(C_7)={\rm
reg}(R/I(C_7))\leq\dim(R/I(C_7))=\beta_0(C_7)=3.
$$
\quad The neighbor set of $A=\{t_1,t_4\}$ in $C_7$ is
$N_{C_7}(A)=\{t_2,t_3,t_5,t_7\}$ and $N_{C_7}(A)$ is
a minimal vertex cover of $C_7$, that is, $A\in\mathcal{A}_{C_7}$. 
Using \textit{Macaulay}$2$ \cite{mac2} and
Procedure~\ref{procedure1} we obtain ${\rm ht}(I(T_{10}))=\alpha_0(G)=6$, ${\rm
pd}(S/I(T_{10}))=7$ and 
$$
{\rm v}(I(T_{10}))=2={\rm im}(T_{10})\leq{\rm
reg}(S/I(T_{10}))=3\leq\dim(S/I(T_{10}))=\beta_0(T_{10})=4.
$$
\quad The neighbor set of $A=\{t_1,t_4\}$ in $T_{10}$ is
$N_{T_{10}}(A)=\{t_2,t_3,t_5,t_7,t_8,t_{10}\}$ and $N_{T_{10}}(A)$ is
a minimal vertex cover of $T_{10}$, that is, $A\in\mathcal{A}_{T_{10}}$.

\begin{figure}[H]
%\centering
\begin{tikzpicture}[line width=.5pt,scale=0.75]
		\tikzstyle{every node}=[inner sep=1pt, minimum width=5.5pt] 
\tiny{
\node (1) at (-11,1.5){$\bullet$};
\node (2) at (-9,0.8) {$\bullet$};
\node (3) at (-8.5,-0.5) {$\bullet$};
\node (4) at (-10,-1.5){$\bullet$};
\node (7) at (-13,0.8){$\bullet$};
\node (5) at (-12,-1.5) {$\bullet$};
\node (6) at (-13.5,-0.5) {$\bullet$};
\node at (-11,1.8){$t_{1}$};
\node at (-9,1.1) {$t_{2}$};
\node at (-8.2,-0.5){$t_{3}$};
\node at (-10,-1.8){$t_{4}$};
\node at (-13,1.1) {$t_{7}$};
\node at (-12,-1.8) {$t_{5}$};
\node at (-13.8,-0.5) {$t_{6}$};
\node at (-11,-2.5) {{\large $C_{7}$}};
\draw[-,line width=1pt] (1) to (2);
\draw[-,line width=1pt] (2) to (3);
\draw[-,line width=1pt] (3) -- (4);
\draw[-,line width=1pt] (4) -- (5);
\draw[-,line width=1pt] (5) -- (6);
\draw[-,line width=1pt] (6) -- (7);
\draw[-,line width=1pt] (7) -- (1);
}
\tiny{
\node (1) at (0,2){$\bullet$};
\node (2) at (2,0.8) {$\bullet$};
\node (3) at (2,-0.5) {$\bullet$};
\node (4) at (2,-2){$\bullet$};
\node (7) at (-2,0.8){$\bullet$};
\node (5) at (-2,-2) {$\bullet$};
\node (6) at (-2,-0.5) {$\bullet$};
\node (8) at (0,1){$\bullet$};
\node (9) at (0,0) {$\bullet$};
\node (10) at (0,-1) {$\bullet$};
\node at (-0,2.3){$t_{1}$};
\node at (2,1.1) {$t_{2}$};
\node at (2.3,-0.5){$t_{3}$};
\node at (2,-2.3){$t_{4}$};
\node at (-2,1.1) {$t_{7}$};
\node at (-2,-2.3) {$t_{5}$};
\node at (-2.3,-0.5) {$t_{6}$};
\node at (0.3,1) {$t_{8}$};
\node at (0.3,0) {$t_{9}$};
\node at (0,-1.3) {$t_{10}$};
\node at (-0,-2.8) {{\large $T_{10}$}};
\draw[-,line width=1pt] (1) to (2);
\draw[-,line width=1pt] (2) to (3);
\draw[-,line width=1pt] (3) -- (4);
\draw[-,line width=1pt] (4) -- (5);
\draw[-,line width=1pt] (5) -- (6);
\draw[-,line width=1pt] (6) -- (7);
\draw[-,line width=1pt] (7) -- (1);
\draw[-,line width=1pt] (1) -- (8);
\draw[-,line width=1pt] (8) -- (9);
\draw[-,line width=1pt] (9) -- (10);
\draw[-,line width=1pt] (5) -- (10);
\draw[-,line width=1pt] (4) -- (10);
}
\end{tikzpicture}
\caption{Two well-covered graphs with no $4$- or
$5$-cycles\quad\quad\quad\quad\quad}\label{C7-T10}
\end{figure}
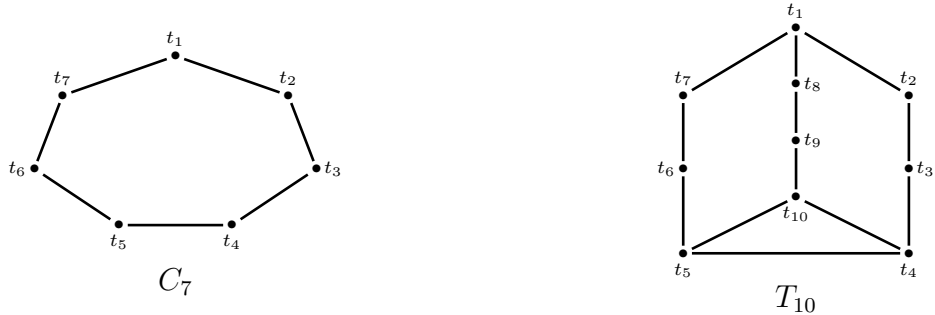
\end{example}

\begin{example}\label{example-w2}
Let $G$ be the graph consisting of two disjoint $3$-cycles with
vertices $x_1,x_2,x_3$ and $y_1,y_2,y_3$. Take two disjoint
independent sets of $G$, say $A_1=\{x_1\}$ and $A_2=\{y_1\}$, to
verify that $G$ is a graph in $W_2$ note that $B_1=\{x_1,y_2\}$
and $B_2=\{y_1,x_2\}$ are maximum independent sets of $G$ containing 
$A_1$ and $A_2$ and $|B_i|=\beta_0(G)=2$.
\end{example}

\begin{appendix}

\section{Procedures}\label{Appendix}

\begin{procedure}\label{procedure1}
Computing the v-number and other invariants of a graded ideal $I$
with \textit{Macaulay}$2$ \cite{mac2}.  This procedure corresponds to
Example~\ref{example1}. One can compute other examples by changing the
polynomial ring $S$ and the generators of the ideal $I$.
\begin{verbatim}
S=QQ[t1,t2,t3]
I=ideal(t1^5,t2^5,t2^4*t3^5,t1^4*t3^5)
--This gives the dimension and the height of I
--If I=I(G), G a graph, this gives the stability 
--number and the covering number of G 
dim(I), codim I
--This gives the associated primes of I
--If I=I(G), this gives the minimal vertex covers of G 
L=ass I
--This determines whether or not I has embedded primes
--If I=I(G), this determines whether or not G is well covered
apply(L,codim)
p=(n)->gens gb ideal(flatten mingens(quotient(I,L#n)/I))
--This computes a minimal generating set for (I:p)/I
MG=(n)->flatten entries  p(n)
MG(0), MG(1)
--This gives the list of all minimal generators g of
--(I:p)/I such that (I: g)=p   
F=(n)->apply(MG(n),x-> if not quotient(I,x)==L#n then 0 
 else x)-set{0}
F(0), F(1)
--This computes the v-number of a graded ideal I
vnumber=min flatten degrees ideal(flatten apply(0..#L-1,F))
M=coker gens gb I
regularity M
--This gives the projective dimension of S/I
pdim M
\end{verbatim}
\end{procedure}

\end{appendix}

%\noindent {\bf Acknowledgments.} We thank the referee for a careful
%reading of the paper and for the improvements suggested. 

\section*{Acknowledgments} 
We used \textit{Macaulay}$2$ \cite{mac2} 
to implement the algorithm to compute the v-number of
graded ideals and to compute other algebraic invariants. 

%\section*{\ } On behalf of all authors, the corresponding 
%author states that there is no conflict of interest.

\bibliographystyle{plain}

\end{document}